\documentclass[11pt,review]{article}
\usepackage{amsmath,amsthm,amssymb,amscd}
\usepackage{graphicx}
\usepackage{graphics}
\usepackage{verbatim}

\usepackage{mathrsfs}
\usepackage{color}
\usepackage[top=1in, bottom=1in, left=1.25in, right=1.25in]{geometry}
\usepackage{enumerate}

\newtheorem{theorem}{Theorem}[section]
\newtheorem{proposition}[theorem]{Proposition}

\newtheorem{lemma}[theorem]{Lemma}
\newtheorem{definition}[theorem]{Definition}

\newtheorem{claim}{Claim}[theorem]

\begin{document}

\title{Circular coloring of  signed graphs}

\vspace{3cm}

\author{Yingli Kang \thanks{Fellow of the International Graduate School ``Dynamic Intelligent Systems''; yingli@mail.upb.de}, Eckhard Steffen\thanks{
		Paderborn Institute for Advanced Studies in
		Computer Science and Engineering,
		Paderborn University,
		Warburger Str. 100,
		33098 Paderborn,
		Germany;			
		es@upb.de}}

\date{}

\maketitle

\begin{abstract}
Let $k, d$ ($2d \leq k)$ be two positive integers. We generalize the well studied notions of $(k,d)$-colorings and of the 
circular chromatic number $\chi_c$ to signed graphs. This implies a new notion of colorings of signed graphs,
and the corresponding chromatic number $\chi$.  
Some basic facts on circular colorings of signed graphs and on the circular chromatic number are proved, and differences to
the results on unsigned graphs are analyzed. In particular, we show that the difference between the circular chromatic number and the chromatic number of
a signed graph is at most 1. Indeed, there are signed graphs where the difference is 1. On the other hand,
for a signed graph on $n$ vertices, if the difference is smaller than 1, then there exists $\epsilon_n>0$, such that
the difference is at most $1 - \epsilon_n$.

We also show that notion of $(k,d)$-colorings is equivalent to $r$-colorings (see \cite{Zhu_2005} 
(X.~Zhu, Recent developments in circular coloring of graphs, in Topics in Discrete Mathematics
Algorithms and Combinatorics Volume {\bf 26}, Springer Berlin Heidelberg (2006) 497-550)).
\end{abstract}

\section{Introduction} \label{Intro}
Graphs in this paper are simple and finite.  The vertex set of a graph $G$ is denoted by $V(G)$, and the edge set by $E(G)$.
A signed graph $(G,\sigma)$ is a graph $G$ and a function $\sigma : E(G) \rightarrow \{ \pm 1 \}$,
which is called a signature of $G$. The set $N_{\sigma} = \{e : \sigma(e) = -1\}$ is the set of negative edges of $(G,\sigma)$ and 
$E(G) - N_{\sigma}$ the set of positive edges. For $v \in V(G)$, let $E(v)$ be the set of edges which are incident to $v$. 
A switching at $v$ defines a graph $(G,\sigma')$ with $\sigma'(e) = -\sigma(e)$ for $e \in E(v)$ and $\sigma'(e) = \sigma(e)$ otherwise. 
Two signed graphs $(G,\sigma)$ and $(G,\sigma^*)$ are {\em equivalent} if they can be obtained from each other by a sequence of switchings. We also say that
$\sigma$ and $\sigma^*$ are equivalent signatures of $G$. 

A circuit in $(G,\sigma)$ is balanced, if it contains an even number of negative edges; otherwise it is unbalanced. The graph
$(G,\sigma)$ is unbalanced, if it contains an unbalanced circuit; otherwise $(G,\sigma)$ is balanced. 
It is well known (see e.g.~\cite{Raspaud_Zhu_2011})
that $(G,\sigma)$ is balanced if and only if it is equivalent to the signed graph with no negative edges, and $(G,\sigma)$ is antibalanced if
it is equivalent to the signed graph with no positive edges. 
Note, that a balanced bipartite graph is also antibalanced. The underlying unsigned graph of $(G,\sigma)$ is denoted by $G$. 

In the 1980s Zaslavsky \cite{Zaslavsky_1982, Zaslavsky_1982_2, Zaslavsky_1984} started studying vertex colorings of signed graphs. 
The natural constraints for a coloring $c$ of a signed graph $(G,\sigma)$ are, that $c(v) \not= \sigma(e) c(w)$ for each edge $e=vw$,
and that the colors can be inverted under switching, i.e.~equivalent signed graphs have the same chromatic number. 
In order to guarantee these properties of a coloring, 
Zaslavsky \cite{Zaslavsky_1982} used the set $\{-k, \dots, 0, \dots ,k\}$ of $2k+1$ "signed colors" and studied the 
interplay between colorings and zero-free colorings through the chromatic polynomial.

Recently, M\'a\v{c}ajov\'a, Raspaud, and \v{S}koviera \cite{MRS_2014}
modified this approach. If $n = 2k+1$, then let $M_n = \{0, \pm 1, \dots,\pm k\}$, and
if  $n = 2k$, then let $M_n = \{\pm 1, \dots,\pm k\}$. A mapping $c$ from $V(G)$ to $M_n$
is a $n$-coloring of $(G,\sigma)$, if $c(v) \not= \sigma(e) c(w)$ for each edge $e=vw$. They 
defined $\chi_{\pm}((G,\sigma))$ to be the smallest number $n$ such that $(G,\sigma)$ has a $n$-coloring.

Since every element of an additive abelian groups has an inverse element, it is natural to choose the elements of an additive abelian group as colors for a coloring of signed
graphs. The self-inverse elements of the group play a crucial role in the colorings, since the induced color classes are independent sets. Hence, the 
following statement is true.

\begin{proposition} Let $G$ be a  graph and $\chi(G) = k$. If ${\cal C}$ is a set of $k$ pairwise different self-inverse elements of an abelian
group (e.g.~of $\mathbb{Z}_2^n$ $(k \leq 2^n)$), then every $k$-coloring of $G$ with colors from ${\cal C}$  is a $k$-coloring of $(G,\sigma)$, for every signature $\sigma$ of $G$. 
In particular, the chromatic number of $(G,\sigma)$ with respect to ${\cal C}$ is $k$. 
\end{proposition}

\subsection{$(k,d)$-colorings of signed graphs}

A coloring parameter, where the colors are also the elements of an abelian group, namely the cyclic group of integer modulo $n$, 
and where the coloring properties are defined by using operations within the group, is the circular chromatic number. This 
parameter was introduced by Vince \cite{Vince_1988} in 1988. We combine these two approaches to define the circular chromatic number of a signed graph. 
For $x \in \mathbb{R}$ and a positive real number $r$, we denote by $[x]_r$, the remainder of $x$ divided by $r$, and define $|x|_r = \min\{[x]_r, [-x]_r\}$. Hence, $[x]_r \in [0,r)$ and $|x|_r=|-x|_r$.

Let $\mathbb{Z}_n$ denote the cyclic group of integers modulo $n$, $\mathbb{Z}/n\mathbb{Z}$.
Let $k$ and $d$ be positive integers such that $k \geq 2d$. A $(k,d)$-{\em coloring} of a signed graph $(G,\sigma)$ is a 
mapping $c: V(G) \mapsto \mathbb{Z}_k$ such that for each edge $e$ with $e=vw$: $d \leq |c(v) - \sigma(e)c(w)|_k$. 
The {\em circular chromatic number} $\chi_c((G,\sigma))$ is $\inf\{\frac{k}{d} : (G,\sigma) \mbox{ has a } (k,d)\mbox{-coloring}\}$.
The minimum $k$ such that $(G,\sigma)$ has a $(k,1)$-coloring is the chromatic number of $(G,\sigma)$ and
it is denoted by $\chi((G,\sigma))$.

\begin{proposition} \label{switching}
Let $k, d$ be positive integers, $(G,\sigma)$ be a signed graph and $c$ be a $(k,d)$-coloring of $(G,\sigma)$. If $(G,\sigma)$ and $(G,\sigma')$ are equivalent,
then there is a $(k,d)$-coloring $c'$ of $(G,\sigma')$. In particular, $\chi_c((G,\sigma)) = \chi_c((G,\sigma'))$.
\end{proposition}

\begin{proof}
Let $x \in V(G)$ and $(G,\sigma')$ be obtained from $(G,\sigma)$ by a switching at $x$. Define $c' : V(G) \rightarrow \mathbb{Z}_k$ with $c'(v) = c(v)$, if $v \not = x$, and $c'(x) = - c(x)$. 
For every edge $e$ with $e=uw$: If $x \not \in \{u,w\}$, then $|c(u) - \sigma(e)c(w)|_k = |c'(u) - \sigma'(e)c'(w)|_k$, and if $x \in \{u,w\}$, say $x=w$, then 
$|c'(u) - \sigma'(e)c'(w)|_k =  |c(u) - (-\sigma(e))(-c(w))|_k = |c(u) - \sigma(e)c(w)|_k$. Hence, $c'$  is a $(k,d)$-coloring of $(G,\sigma')$, and therefore,
$\chi_c((G,\sigma)) = \chi_c((G,\sigma'))$.             
\end{proof}

Note, that if $(G,\sigma)$ has a $(k,d)$-coloring, then by switching we can obtain an equivalent graph $(G,\sigma')$ and a $(k,d)$-coloring $c'$ on $(G,\sigma')$ such that
$c'(v) \in \{0, 1, \dots, \lfloor \frac{k}{2} \rfloor\}$ for each $v \in V(G)$. 
We will show that the circular chromatic number is a minimum; i.e.~if $\chi_c((G,\sigma))=\frac{k}{d}$, then there exists a $(k,d)$-coloring of $(G,\sigma)$. 
Furthermore, $\chi((G,\sigma)) - 1 \leq \chi_c((G,\sigma)) \leq \chi((G,\sigma))$ for $(G,\sigma)$. In contrast to the corresponding result for unsigned graphs we show, that
for each even $k$ there are signed graphs with circular chromatic number $k$ and chromatic number $k+1$, i.e.~they do not have a $(k,1)$-coloring.
On the other hand, for a signed graph on $n$ vertices, if the difference between these parameters is smaller than 1, then there exists $\epsilon_n>0$, such that
the difference is at most $1 - \epsilon_n$.
The proofs of the main results of this paper follow the approach of Bondy and Hell \cite{Bondy_Hell_1990} for similar results for unsigned graphs.

\subsection{$r$-colorings of signed graphs}

The name "circular coloring" was introduced by Zhu \cite{Zhu_1992}, and motivated by the equivalence of $(k,d)$-colorings to $r$-colorings. 
In section \ref{equivalence} we will show that this is also true in the context of signed graphs. 
Let $(G,\sigma)$ be a signed graph and $r$ be a real number at least 1. A {\em circular} $r$-{\em coloring} of $(G,\sigma)$
is a function $f : V(G) \rightarrow [0,r)$ such that for any edge $e$ with $e=xy$: 
if $\sigma(e) = 1$, then $1 \leq |f(x) - f(y)| \leq r-1$, and 
if $\sigma(e) = -1$, then $1 \leq |f(x)+ f(y) - r| \leq r-1$.
Clearly, if we identify 0 and $r$ of the interval $[0,r]$ into a single point, then we obtain a circle with perimeter $r$. Let $S^r$ be this circle. 
The colors are the points on $S^r$, and the distance 
between two points $a, b$ of $S^r$ is the shorter arc of $S^r$ connecting $a$ and $b$, which is $|a-b|_r$.  
For $a \in S^r$ let $ r-a$ be the inverse element of $a$. By this notation, a circular $r$-coloring of $(G,\sigma)$ is a function that assigns 
if $\sigma(e) = 1$, then $1 \leq |f(x) - f(y)|_r$, and if $\sigma(e) = -1$, then $1 \leq |f(x)+ f(y)|_r$. 
Note, that this definition also respects switchings. Let $f$ be an $r$-coloring of $(G,\sigma)$ and $(G,\sigma')$ be obtained from $(G,\sigma)$ by a switching at $v \in V(G)$. Then
$f'$ with $f'(x) = f(x)$ if $x \in V(G) \setminus \{v\}$ and $f'(v) = r-f(v)$ is an $r$-coloring of $(G,\sigma')$. 
As above we deduce that there is always a coloring on an equivalent graph of $(G,\sigma)$, which only uses colors in the interval $[0,\frac{r}{2}]$.

We will show in section \ref{equivalence} that $\chi_c((G,\sigma)) = \inf\{r : (G,\sigma) \mbox{ has an } r\mbox{-coloring}\}$.

The circular chromatic number and $r$-colorings seem to be a very natural notion for the coloring of signed graphs. 
The color set $S^r$ has always two self-inverse elements, namely $0$ and $\frac{r}{2}$.

\section{Basic properties of $(k,d)$-colorings and $\chi_c$} \label{basic_facts}

\begin{lemma}\label{gcd at least 3} Let $d, k, t$ be positive integers with $gcd(k,d)=1$ and $t\geq 3$, and let $(G,\sigma)$ be a signed graph. 
If $(G,\sigma)$ has a $(tk,td)$-coloring, then it has a $(tk-2k,td-2d)$-coloring.
\end{lemma}

\begin{proof}
For $i\in\{0,1,\cdots, t-1\}$, let $A_i=\{i, i+t, i+2t,\cdots, i+(k-1)t\}$.
Clearly, $A_0,\ldots,A_{t-1}$ are $t$ pairwise disjoint sets of colors whose union is exactly the color set $\mathbb{Z}_{tk}$.
We shall recolor each color in both sets $A_1$ and $A_{t-1}$ by a color in set $A_0$ as follows: for $i\in A_1$, recolor $i$ by $i-1$, and for $i\in A_{t-1}$, recolor $i$ by $i+1$. We obtain a 
new $(tk,td)$-coloring of $(G,\sigma)$ in which no vertex receives a color from $A_1\cup A_{t-1}$. Define $k'=tk-2k$. Since the colors in the set $A_1\cup A_{t-1}$ are not used, we define a new coloring by renaming colors by elements of $\mathbb{Z}_{k'}$. Change 
color $x$ (from $\mathbb{Z}_{tk}$) to $x-|\{y\colon\ y\in A_1\cup A_{t-1}~\text{and}~y<x\}|$ (interpreted as element in $\mathbb{Z}_{k'}$) to obtain a mapping $\phi': V\rightarrow \mathbb{Z}_{k'}$. 
Let $d'=td-2d$. We claim that $\phi'$ is a $(k',d')$-coloring of $(G,\sigma)$.
Denote by $I_j$ the set $\{j,j+1,\ldots, j+td-1\}$ which is an interval of $\mathbb{Z}_{tk}$. Each interval $I_j$ contains exactly $2d$ elements of $A_1\cup A_{t-1}$, and any pair 
of mutually inverse elements of $\mathbb{Z}_{tk}$ has been recolored by a pair of mutually inverse elements of $\mathbb{Z}_{k'}$. It follows that $\phi'$ is a $(k',d')$-coloring of $(G,\sigma)$, as required.
\end{proof}

By the rearrangement theorem of group theory we have

\begin{lemma}\label{algebra}
Let $k,d$ and $x$ be three integers with $k,d>0$ and $gcd(k,d)=1$.
If $A=\{0,1,\ldots,k-1\}$ and
$B=\{[x+id]_k\colon\ i\in A\}$,
then $A=B$.
\end{lemma}

\begin{definition} \label{update}
Let $c$ be a $(k,d)$-coloring of a signed graph $(G,\sigma)$ in which colors $x_0$ and its inverse $k-x_0$ are missing.
{\bf Updating $c$ at $x_0$} is defined as follows: if the color $[x_0+d]_k$ appears in $c$, then recolor $[x_0+d]_k$ by $[x_0+d-1]_k$; and meanwhile, if the color $[k-x_0-d]_k$ appears in $c$, then recolor $[k-x_0-d]_k$ by $[k-x_0-d+1]_k$. 
Let $r$ be a positive integer.
Updating $c$ at a sequence of colors $x_0,[x_0+d]_k,\ldots,[x_0+(r-1)d]_k$ is called \emph{updating $c$ from $x_0$ by $r$ steps}. We also say that a function $c'$ is obtained from $c$
by updating at $x_0$ (in $r$ steps) if $c'$ is the final function from $V(G)$ to $\mathbb{Z}_{k}$ in this process.
\end{definition}

Let $k, d$ be two positive integers and $P(k,d)=\{\frac{1}{2}(k-2d+1), \frac{1}{2}(k-d+1), \frac{1}{2}(2k-d+1)\}$. Clearly, 
if both $k$ and $d$ are even, then $\mathbb{Z}_k\cap P(k,d)=\emptyset$; otherwise, $|\mathbb{Z}_k\cap P(k,d)|=2$.

\begin{lemma} \label{P(k,d)}
	Let $(G,\sigma)$ be a signed graph, $c$ be a $(k,d)$-coloring of $(G,\sigma)$, and let $c'$ be obtained from $c$ by updating at $x_0$. 
	Either $x_0\notin P(k,d)$ or both $[x_0+d]_k$ and $[k-x_0-d]_k$ are not used in $c$ if and only if $c'$ is a $(k,d)$-coloring of $(G,\sigma)$ in which the colors $x_0, [x_0+d]_k, [k-x_0]_k$ and $[k-x_0-d]_k$ are not used.
\end{lemma}

\begin{proof}
	$(\Rightarrow)$	
	If both $[x_0+d]_k$ and $[k-x_0-d]_k$ are not used in $c$, then
	it follows that $c'$ is the same coloring as $c$ since nothing happens in the updating process. So we are done in this case.
	
	Let $x_0\notin P(k,d)$, and suppose to the contrary that $c'$ is not a $(k,d)$-coloring of $(G,\sigma)$. Then there exists an edge $e$ with two end-points $u$ and $v$ such that $|c'(u)-\sigma(e)c'(v)|_k<d.$
	Since $c$ is a $(k,d)$-coloring of $(G,\sigma)$, it follows that $|c(u)-\sigma(e)c(v)|_k\geq d.$
	Hence, the distance between the colors of $u$ and $v$ has been decreased in the updating process.
	The distance can be decreased by at most 2. Hence, we distinguish two cases.
	
	Case a: The distance between the colors of $u$ and $v$ decreases by 2. 
	In this case, both $u$ and $v$ have been recolored, say $c(u)=[x_0+d]_k$ and $c(v)=[k-x_0-d]_k;$ and moreover, 
	$[c(u)-\sigma(e)c(v)]_k\in \{d,d+1\}.$ 
	It follows that $\sigma(e)=1$ and furthermore,
	$[c(u)-\sigma(e)c(v)]_k=d+1$ since for otherwise $c(u)$ and $c(v)$ are in fact the colors $x_0$ and $[k-x_0]_k$ which are missing in $c$.
	By simplification of this equality, we get
	$[2(x_0+d)-k]_k=d+1$ and thus, $x_0\in\{\frac{k-d+1}{2}, \frac{2k-d+1}{2}\}$, contradicting the assumption that $x_0\notin P(k,d).$
	
	Case b: The distance between the colors of $u$ and $v$ decreases by 1. In this case,  exactly one of $u$ and $v$ has been recolored, say $u$; and moreover, $|c(u)-\sigma(e)c(v)|_k=d.$ Without loss of generality, we may assume $c(u)=[x_0+d]_k$. It follows that $c(v)=x_0$, contradicting the fact that $x_0$ is not used in $c$.
	
	Therefore, $c'$ is a $(k,d)$-coloring of $(G,\sigma).$
	If the colors $[x_0+d]_k$ and $[k-x_0-d]_k$ occur in $c'$, then they have been recolored by each other, which can happen in the only case that $k$ is odd and $[x_0+d]_k=x_0+d=\frac{k+1}{2}$. However, this case is impossible since $x_0\notin P(k,d)$.
	Finally, suppose to the contrary that the colors $x_0$ and $[k-x_0]_k$ occur in $c'$. Since they are not used in $c$, they have been reused in the updating process. Thus, $[x_0+d-1]_k=[k-x_0]_k$ and so $x_0\in\{\frac{k-d+1}{2}, \frac{2k-d+1}{2}\}$, a contradiction.
	
	$(\Leftarrow)$
	Suppose to the contrary that $x_0\in P(k,d)$ and at least one of $[x_0+d]_k$ and $[k-x_0-d]_k$ are used in $c$. Without loss of generality, say $[x_0+d]_k$ is used. We distinguish two cases according to the value of $x_0$.
	
	Case 1: assume $x_0\in \{\frac{k-d+1}{2}, \frac{2k-d+1}{2}\}$. Thus, $[x_0+d-1]_k=[k-x_0]_k$, which implies that the color $[k-x_0]_k$ has been reused in the updating process, a contradiction.
	
	Case 2: assume $x_0=\frac{k-2d+1}{2}$. Thus, $[x_0+d]_k=[k-x_0-d]_k+1$, which implies that the colors $[x_0+d]_k$ and $[k-x_0-d]_k$ have been exchanged, a contradiction.
\end{proof}

\begin{lemma}\label{gcd =2}
Let $(G,\sigma)$ be a signed graph on $n$ vertices that has a $(2k,2d)$-coloring and $gcd(k,d)=1$. If $k>2n$, then $(G,\sigma)$ has a $(k,d)$-coloring.
\end{lemma}

\begin{proof}
Let $c$ be a $(2k,2d)$-coloring of $(G,\sigma)$.
Since $k>2n$, we may assume that there is odd $x_0$, such that $x_0$ and $k-x_0$ are not used in $c$. 
Update $c$ from $x_0$ by $k$ steps to obtain a function $c'$. Denote by $A$ the set of odd elements of $\mathbb{Z}_{2k}$.
Since both $2k$ and $2d$ are even it follows with Lemma \ref{algebra} that the colors of $A \cap \{c(v) : v \in V(G)\}$ have been 
recolored by colors of $\mathbb{Z}_{2k}\setminus A$ in the updating process. Hence, $A \cap \{c'(v) : v \in V(G)\} = \emptyset$,
and by Lemma \ref{P(k,d)}, $c'$ is a $(2k,2d)$-coloring of $(G,\sigma)$. Thus,
$\phi: V(G)\rightarrow \mathbb{Z}_{k}$ with $\phi(v) = \frac{1}{2}c'(v)$ is a coloring of $(G,\sigma)$. 
Let $I_j = \{j,j+1,\ldots, j+2d-1\}$ which is an interval of $\mathbb{Z}_{2k}$. Each interval $I_j$ contains exactly $d$ elements of $A$. 
Moreover, any pair of mutually inverse elements of $\mathbb{Z}_{2k}$ has been recolored 
by a pair of mutually inverse elements of  $\mathbb{Z}_{k}$. Hence, $\phi$ is a $(k,d)$-coloring of $(G,\sigma)$, as required.
\end{proof}

\begin{lemma}\label{gcd=1}
If $(G,\sigma)$ is a signed graph on $n$ vertices that has a $(k,d)$-coloring with $gcd(k,d)=1$ and $k>4n$, then $(G,\sigma)$ has a $(k',d')$-coloring with $k'<k$ and $\frac{k'}{d'}< \frac{k}{d}$.
\end{lemma}

\begin{proof}
Since $gcd(k,d)=1$, we may assume that $P(k,d)\cap \mathbb{Z}_k=\{p,q\}$ and $p<q$.

Let $f\colon\ \mathbb{Z}_k\mapsto \mathbb{Z}_k$ such that $x\equiv f(x)d~(\text{mod~} k).$ Lemma \ref{algebra} implies that $f$ is a bijection. Further,
$x$ and $y$ are mutually inverse elements of $\mathbb{Z}_k$ if and only if $f(x)$ and $f(y)$ are mutually inverse ones, and $|f(p)-f(q)|_k=\lfloor\frac{k}{2}\rfloor$.

Let $c$ be a $(k,d)$-coloring of $(G,\sigma)$.
Since $k> 4n$ we may assume that  $x_0 \in \mathbb{Z}_{k}$ such that $x_0$ and $k-x_0$ are not used in $c$, and that 
$f(q), f(x_0)$ and $f(p)$ are in clockwise order in $\mathbb{Z}_k$.

Hence, $c$ can be updated from $x_0$ by $[f(p)-f(x_0)]_k$ steps to obtain a $(k,d)$-coloring $c'$ of $(G,\sigma)$ in which colors $p$ and $k-p$ are not used. 
Let $r=\min \{[f(p)-f(k-p)]_k,[f(q)-f(k-p)]_k\}$, i.e., $r$ is the minimum positive integer such that either $k-p+rd\equiv p ~(\text{mod~} k)$ or $k-p+rd\equiv q ~(\text{mod~} k)$. Updating $c'$ from $k-p$ by $r$ steps, we obtain a function $c''$, which is a $(k,d)$-coloring $c''$ of $(G,\sigma)$ by Lemma \ref{P(k,d)}. 

We will show that no color is reused in this updating process such that we can define a $(k',d')$-coloring of $(G,\sigma)$ with $\frac{k'}{d'} < \frac{k}{d}$ and $k' < k$.  

Let $A=\{[k-p+id]_k\colon\ 0\leq i \leq r\}$ and $B=\{k-a\colon\ a\in A\}$. 
By simplifying the congruence expressions, we reformulate the minimality of $r$ as: $r$ is the minimum positive integer such that\\
(1) either $(r+1)d\equiv 1 ~(\text{mod~} k)$ or $(2r+2)d\equiv 2 ~(\text{mod~} k)$, if $k$ is even;\\
(2) either $(r+2)d\equiv 1 ~(\text{mod~} k)$ or $(2r+3)d\equiv 2 ~(\text{mod~} k)$, if $k$ is odd.

\begin{claim} \label{no_reuse} No element of $A\cup B$ is used in coloring $c''$. 
\end{claim}

Suppose to the contrary that $A\cup B$ has a color $\alpha$ with $\alpha \in\{[k-p+r_1d]_k,k-[k-p+r_1d]_k\}$ appearing in $c''$. Concerning that the color $\alpha$ is missing in the resulting coloring after exactly $r_1$ steps in the updating process, its appearance in $c''$ yields that it has been reused in some $r_2$ step with $r>r_2>r_1$.
It follows that either $k-p+r_2d\equiv k-p+r_1d+1 ~(\text{mod~} k)$ or $k-p+r_2d\equiv -(k-p+r_1d)+1 ~(\text{mod~} k)$.

In the former case, the congruence expression can be simplified as $(r_2-r_1)d\equiv 1 ~(\text{mod~} k)$. Note that $0<r_2-r_1<r+1$. A contradiction is obtained by the minimality of $r$.

In the latter case, the congruence expression can be simplified as $(r_1+r_2+1)d\equiv 2 ~(\text{mod~} k)$ if $k$ is even and
 $(r_1+r_2+2)d\equiv 2 ~(\text{mod~} k)$ if $k$ is odd. But then $r_1+r_2<2r+1$ which is a contradiction to the minimality of $r$.
This completes the proof of the claim.

\textbf{Case 1:} $k$ is even. In this case, $p=\frac{1}{2}(k-d+1)$ and $q=\frac{1}{2}(2k-d+1)$.

Case 1.a: $k-p+rd\equiv p ~(\text{mod~} k)$.

The colors $[k-p+id]_k$ and $[k-p+(r-i)d]_k$ are mutually inverse, for $0\leq i\leq r$. Thus, the set $A$ consists of $\lceil \frac{r+1}{2} \rceil$ pairs of mutually inverse elements of $\mathbb{Z}_k$ and $\{0,\frac{k}{2}\}\nsubseteq A=B$. Since the colors of $A\cup B$ are not used in $c''$ by Claim \ref{no_reuse}, we rename the other colors: if $0\notin A$, then change color $x$ to $x-|\{y\colon\ y\in A~\text{and}~y<x\}|$; otherwise, change color $x$ to $x-|\{y\colon\ y\in A~\text{and}~y<x\}|-\lfloor \frac{k-|A|}{2} \rfloor$. Define $k'=k-r-1$. We thereby obtain a mapping $\phi': V(G) \rightarrow \mathbb{Z}_{k'}$. Denote by $I_j$ the set $\{j,j+1,\ldots, j+d-1\}$ which is an interval of $\mathbb{Z}_k$. Each interval $I_j$ contains at most $\frac{rd+d-1}{k}$ elements of $A$. Define $d'=d-\frac{rd+d-1}{k}$. Moreover, any pair of mutually inverse colors of $\mathbb{Z}_k$ has been recolored to mutually inverse colors of $\mathbb{Z}_k$ and then has been renamed to be mutually inverse colors of $\mathbb{Z}_{k'}$. Hence, $\phi'$ is a $(k',d')$-coloring of $(G,\sigma)$, and $\frac{k'}{d'}=\frac{k(k-r-1)}{d(k-r-1)+1}<\frac{k}{d}$.

Case 1.b: $k-p+rd\equiv q~(\text{mod~} k)$. 

We have that either $0 < f(q), f(k-p) < \frac{k}{2}$ or  $\frac{k}{2} < f(q), f(k-p) < k$.
Since $|f(p)-f(q)|_k=\frac{k}{2}$, it follows that neither $\{f(a) : a \in A\}$ nor $A$ contains any pair of mutually inverse colors.  
Thus, $A\cup B$ consists of $r+1$ pairs of mutually inverse colors and $0,\frac{k}{2} \not\in A\cup B$. Define $k'=k-2(r+1)$.
Since the colors in the set $A\cup B$ are not used in $c''$ by Claim \ref{no_reuse}, we may rename the other colors, changing color $x$ to $x-|\{y\colon\ y\in A\cup B~\text{and}~y<x\}|$, thereby obtain a mapping $\phi': V(G) \rightarrow \mathbb{Z}_{k'}$. Denote by $I_j$ the set $\{j,j+1,\ldots, j+d-1\}$ which is an interval of $\mathbb{Z}_k$. Each interval $I_j$ contains at most $\frac{2rd+2d-2}{k}$ elements of $A$. Define $d'=d-\frac{2rd+2d-2}{k}=\frac{(k-2r-2)d+2}{k}$. By repeating the argument as in Case 1.a, we get a $(k',d')$-coloring of $(G,\sigma)$.Furthermore, $\frac{k'}{d'}=\frac{k(k-2r-2)}{d(k-2r-2)+2}<\frac{k}{d}$.

\textbf{Case 2:} $k$ is odd.
In this case, $p=\frac{1}{2}(k-2d+1)$, and $q=\frac{1}{2}(k-d+1)$ when $d$ is even and $q=\frac{1}{2}(2k-d+1)$ when $d$ is odd.

Case 2.a: $k-p+rd\equiv p ~(\text{mod~} k)$.

The colors $[p+id]_k$ and $[p+(r-i)d]_k$ are mutually inverse for $0\leq i\leq r$. Thus, $A=B$ and $A$ consists of $\lceil \frac{r+1}{2} \rceil$ pairs of mutually inverse colors of $\mathbb{Z}_k$. Since the colors of $A\cup B$ are not used in $c''$ by Claim \ref{no_reuse}, we may rename the other colors: if $0\notin A$,
then change $x$ to $x-|\{y\colon\ y\in A~\text{and}~y<x\}|$ for each $x\leq\lfloor \frac{k}{2}\rfloor$ and to $x-|\{y\colon\ y\in A~\text{and}~y<x\}|-1$ for each $x>\lfloor \frac{k}{2}\rfloor$; otherwise, change $x$ to $x-|\{y\colon\ y\in A~\text{and}~y<x\}|-\frac{k-|A|}{2} +1$ for each $x\leq\lfloor \frac{k}{2}\rfloor$ and to $x-|\{y\colon\ y\in A~\text{and}~y<x\}|- \frac{k-|A|}{2}$ for each $x>\lfloor \frac{k}{2}\rfloor$.
The mutually inverse colors $\frac{k-1}{2}$ and $\frac{k+1}{2}$ of $\mathbb{Z}_k$ are not in $A$ and they have been renamed into the same color. Define $k'=k-r-2$. We thereby obtain a mapping $\phi': V\rightarrow \mathbb{Z}_{k'}$. Denote by $I_j$ the set $\{j,j+1,\ldots, j+d-1\}$ which is an interval of $\mathbb{Z}_k$. Define $d^*=\frac{1}{k}(rd+2d-1)$. For each interval $I_j$, if both colors $\frac{k-1}{2}$ and $\frac{k+1}{2}$ belong to $I_j$, then $I_j$ contains at most $d^*-1$ elements of $A$; otherwise, $I_j$ contains at most $d^*$ elements of $A$. Define $d'=d-d^*$.
Moreover, any pair of mutually inverse colors of $\mathbb{Z}_k$ has been recolored to be mutually inverse colors of $\mathbb{Z}_k$ and then has been renamed to be mutually inverse colors of $\mathbb{Z}_{k'}$. Hence, $\phi'$ is a $(k',d')$-coloring of $(G,\sigma)$, and $\frac{k'}{d'}=\frac{k(k-r-2)}{d(k-r-2)+1}<\frac{k}{d}$.

Case 2.b: $k-p+rd\equiv q~(\text{mod~} k)$.

By similar argument as in Case 1.b, we may assume that $A$ contains no mutually inverse colors of $\mathbb{Z}_k$.
Thus, $A\cup B$ consists of $r+1$ pairs of mutually inverse colors and $0\notin A\cup B$.
Since the colors of $A\cup B$ are not used in $c''$ by Claim \ref{no_reuse}, we may rename the other colors:
change $x$ to $x-|\{y\colon\ y\in A~\text{and}~y<x\}|$ for each $x\leq\lfloor \frac{k}{2}\rfloor$ and to $x-|\{y\colon\ y\in A~\text{and}~y<x\}|-1$ for each $x>\lfloor \frac{k}{2}\rfloor$.
The mutually inverse colors $\frac{k-1}{2}$ and $\frac{k+1}{2}$ of $\mathbb{Z}_k$ are not contained in the set $A$ and have been renamed into the same color. Define $k'=k-2r-3$. We thereby obtain a mapping $\phi': V\rightarrow \mathbb{Z}_{k'}$. Denote by $I_j$ the set $\{j,j+1,\ldots, j+d-1\}$ which is an interval of $\mathbb{Z}_k$. Define $d^*=\frac{1}{k}(2rd+3d-2)$. Clearly, $d^*$ is a positive integer because of the assumption of Case 2.b.
For each interval $I_j$, if both colors $\frac{k-1}{2}$ and $\frac{k+1}{2}$ belong to $I_j$, then $I_j$ contains at most $d^*-1$ elements of $A$; otherwise, $I_j$ contains at most $d^*$ elements of $A$. Define $d'=d-d^*$. Any pair of mutually inverse colors of $\mathbb{Z}_k$ has been recolored to be mutually inverse colors of $\mathbb{Z}_k$ and then has been renamed to be mutually inverse colors of $\mathbb{Z}_{k'}$. Hence, $\phi'$ is a $(k',d')$-coloring of $(G,\sigma)$, and $\frac{k'}{d'}=\frac{k(k-2r-3)}{d(k-2r-3)+2}<\frac{k}{d}$.
\end{proof}

\begin{theorem}\label{cor_min}
If $(G,\sigma)$ is a signed graph on $n$ vertices, then
$$\chi_c((G,\sigma)):= \min \{\frac{k}{d}\colon\ (G,\sigma) \text{ has a $(k,d)$-coloring and $k \leq 4n$}\}.$$
\end{theorem}

\begin{proof}
By Lemmas \ref{gcd at least 3}, \ref{gcd =2} and \ref{gcd=1}, if $(G,\sigma)$ has a $(k,d)$-coloring then it has a $(k',d')$-coloring with $k'\leq 4n$ and $\frac{k'}{d'}\leq  \frac{k}{d}.$ Therefore,
$$\chi_c((G,\sigma)):= \inf \{\frac{k}{d}\colon\ (G,\sigma) \text{ has a $(k,d)$-coloring and } k\leq 4n\}.$$
Since the set $\{\frac{k}{d}\colon\ (G,\sigma) \text{ has a $(k,d)$-coloring and } k\leq 4n\}$ is finite, the infimum can be replaced by a minimum.
\end{proof}

\subsection*{Relation between $\chi_c$ and $\chi$}

\begin{lemma}\label{lem_tk td}
If a signed graph $(G,\sigma)$ has a $(k,d)$-coloring, then for any positive integer $t$, $(G,\sigma)$ has a $(tk,td)$-coloring.
\end{lemma}

\begin{proof}
Let $c$ be a $(k,d)$-coloring of $(G,\sigma)$.
Define a $(tk,td)$-coloring $c'$ of $(G,\sigma)$ by
$$c'(x)=tc(x), \text{ for all } x\in V(G).$$
\end{proof}

\begin{lemma}\label{lem_bigger k}
If a signed graph $(G,\sigma)$ has a $(k,d)$-coloring and $k'>k$, where $k'$ is a positive integer, then $(G,\sigma)$ has a $(k',d)$-coloring.
\end{lemma}

\begin{proof}
Let $c$ be a $(k,d)$-coloring of $(G,\sigma)$.
Define the mapping $c': V(G)\rightarrow \mathbb{Z}_{k'}$ by for all $x\in V(G)$,
\begin{eqnarray*}
c'(x)=
\begin{cases}
c(x)  & \text{if } c(x)\leq \lfloor \frac{k}{2} \rfloor,\\
c(x)+k'-k & \text{otherwise}.
\end{cases}
\end{eqnarray*}
It is easy to check that $c'$ is a $(k',d)$-coloring of $(G,\sigma)$.
\end{proof}

\begin{theorem}\label{pro_bigger coloring}
If a signed graph $(G,\sigma)$ has a $(k,d)$-coloring, and $k'$ and $d'$ are two positive integers such that $\frac{k}{d}<\frac{k'}{d'}$, then $(G,\sigma)$ has a $(k',d')$-coloring.
\end{theorem}

\begin{proof}
By Lemma \ref{lem_tk td}, $(G,\sigma)$ has a $(kd',dd')$-coloring.
Since $\frac{k}{d}<\frac{k'}{d'}$, Lemma \ref{lem_bigger k} implies that $(G,\sigma)$ has a $(k'd-1,dd')$-coloring and a $(k'd,dd')$-coloring as well. If $d$ is odd, then by Lemma \ref{gcd at least 3}, a $(k'd,dd')$-coloring of $(G,\sigma)$ yields a $(k',d')$-coloring of $(G,\sigma)$ and we are done. Let $d$ be even and $c''$ be a $(k'd-1,dd')$-coloring of $(G,\sigma)$. 
Define the mapping $c: V(G) \rightarrow \{1-\frac{d}{2},2-\frac{d}{2},\ldots,k'd-1-\frac{d}{2}\}$ as follows. For $x\in V(G)$ let 
\begin{eqnarray*}
c(x)=
\begin{cases}
c''(x)-(k'd-1), &\text{ if } c''(x)>k'd-1-\frac{d}{2},\\
c''(x), & \text{ otherwise. }
\end{cases}
\end{eqnarray*}
Define the mapping $c': V(G) \rightarrow \mathbb{Z}_{k'}$ by
$$c'(x)=\lfloor \frac{c(x)}{d}+\frac{1}{2}\rfloor, \text{ for all } x\in V(G).$$
We will show that $c'$ is a $(k',d')$-coloring of $(G,\sigma)$. 

Consider an edge $uv$.
First assume that $\sigma(uv)=1$. Without loss of generality, let $c(u)>c(v)$. Note that $1 \leq c(u)-c(v)\leq k'd-2$. Since $c''$ is a $(k'd-1,dd')$-coloring of $(G,\sigma)$,
$$dd'\leq c(u)-c(v)\leq k'd-1-dd'.$$
Therefore,
\begin{equation*}
\begin{split}
c'(u)-c'(v)
&=\lfloor \frac{c(u)}{d}+\frac{1}{2}\rfloor-\lfloor \frac{c(v)}{d}+\frac{1}{2}\rfloor\\
& \leq \lfloor k'-d'+\frac{c(v)-1}{d}+\frac{1}{2}\rfloor-\lfloor \frac{c(v)}{d}+\frac{1}{2}\rfloor\\
&\leq k'-d',
\end{split}
\end{equation*}
and
\begin{equation*}
\begin{split}
c'(u)-c'(v)
&=\lfloor \frac{c(u)}{d}+\frac{1}{2}\rfloor-\lfloor \frac{c(v)}{d}+\frac{1}{2}\rfloor\\
& \geq \lfloor d'+\frac{c(v)}{d}+\frac{1}{2}\rfloor-\lfloor \frac{c(v)}{d}+\frac{1}{2}\rfloor\\
&= d'.
\end{split}
\end{equation*}
Next assume that $\sigma(uv)=-1.$ Note that $2-d\leq c(u)+c(v)\leq 2(k'd-1)-d$. Since $c''$ is a $(k'd-1,dd')$-coloring of $(G,\sigma)$,
either
$$dd'\leq c(u)+c(v)\leq k'd-1-dd'$$
or
$$k'd-1+dd'\leq c(u)+c(v)\leq 2(k'd-1)-dd'.$$
In the former case,
\begin{equation*}
\begin{split}
c'(u)+c'(v)
&=\lfloor \frac{c(u)}{d}+\frac{1}{2}\rfloor+\lfloor \frac{c(v)}{d}+\frac{1}{2}\rfloor\\
& \leq \lfloor k'-d'-\frac{c(v)+1}{d}+\frac{1}{2}\rfloor-\lfloor \frac{c(v)}{d}+\frac{1}{2}\rfloor\\
&\leq \lfloor k'-d'-\frac{1}{d}+1\rfloor.\\
&=k'-d',
\end{split}
\end{equation*}
and
\begin{equation*}
\begin{split}
c'(u)+c'(v)
&=\lfloor \frac{c(u)}{d}+\frac{1}{2}\rfloor+\lfloor \frac{c(v)}{d}+\frac{1}{2}\rfloor\\
& \geq \lfloor d'-\frac{c(v)}{d}+\frac{1}{2}\rfloor+\lfloor \frac{c(v)}{d}+\frac{1}{2}\rfloor\\
&=d'.
\end{split}
\end{equation*}
In the latter case, by a similar calculation, we deduce
$$k'+d'\leq c'(u)+c'(v)\leq 2k'-d'.$$
Therefore, $c'$ is a $(k',d')$-coloring of $(G,\sigma)$.
\end{proof}

\begin{proposition}
If a signed graph $(G,\sigma)$ has a $(k,d)$-coloring with $d$ odd, and $k'$ and $d'$ are two positive integers such that $\frac{k}{d}=\frac{k'}{d'}$, then $(G,\sigma)$ has a $(k',d')$-coloring.
\end{proposition}

\begin{proof}
By Lemma \ref{lem_tk td}, $(G,\sigma)$ has a $(kd',dd')$-coloring, i.e., a $(k'd,dd')$-coloring
since $\frac{k}{d}=\frac{k'}{d'}$. Since $d$ is odd, by Lemma \ref{gcd at least 3}, $(G,\sigma)$ has a $(k',d')$-coloring of $(G,\sigma)$.
\end{proof}

\begin{theorem} \label{bounds}
If $(G,\sigma)$ is a signed graph, then $\chi((G,\sigma))-1\leq \chi_c((G,\sigma))\leq \chi((G,\sigma)).$
\end{theorem}

\begin{proof}
By the definitions, we have $\chi_c((G,\sigma))\leq \chi((G,\sigma)).$
On the other hand, suppose to the contrary that $\chi_c((G,\sigma))< \chi((G,\sigma))-1$.
Theorem \ref{cor_min} implies that $\chi_c((G,\sigma))$ is a rational number.
We may assume $(G,\sigma)$ has a $(k,d)$-coloring with $\chi_c((G,\sigma))=\frac{k}{d}$.
By Theorem \ref{pro_bigger coloring}, $(G,\sigma)$ has a $(\chi((G,\sigma))-1,1)$-coloring, a contradiction.
\end{proof}

If $G$ is an unsigned graph, then $\chi(G) - 1 < \chi_c(G) \leq \chi(G)$, see \cite{Vince_1988}. We will show that there are
signed graphs $(G,\sigma)$ with $\chi((G,\sigma))-1 = \chi_c((G,\sigma))$, see Figure \ref{Fig1}.

\begin{figure}[hh]
	\centering
	\includegraphics[width=7cm]{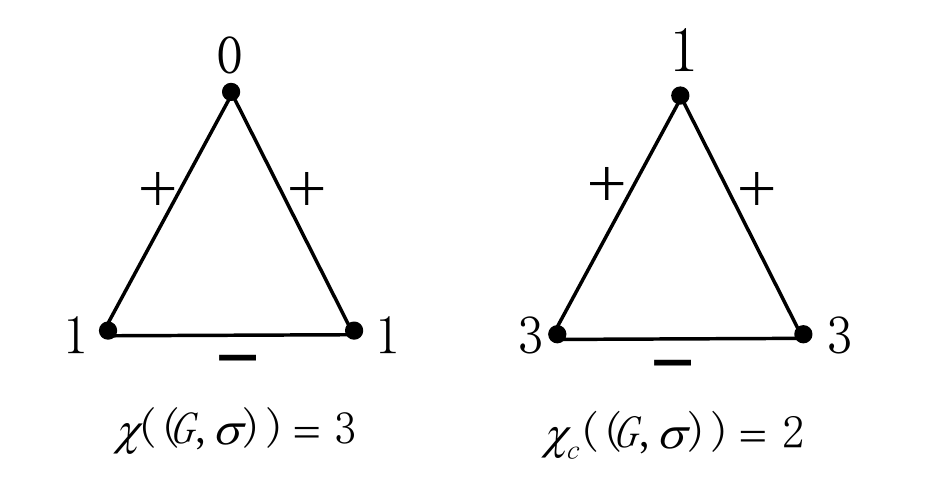}\\
	\caption{An example $\chi-1=\chi_c$}\label{Fig1}
\end{figure}

\begin{theorem} \label{charac_X_c=X-1}
 Let $(G,\sigma)$ be a signed graph with $\chi((G,\sigma))=t+1$. The following statements are equivalent.
\begin{enumerate}
\item $\chi_c((G,\sigma)) = t$.
\item $(G,\sigma)$ has a  $(2t,2)$-coloring.
\end{enumerate}
\end{theorem}

\begin{proof} 
($\Rightarrow$) Let $\chi_c((G,\sigma)) = t$. For each $(k,d)$-coloring of $(G,\sigma)$ with $\frac{k}{d}=t$ it follows that
$d>1$ since for otherwise we would get a $(t,1)$-coloring. If $d$ is odd, then Lemma \ref{gcd at least 3} implies, 
that there is $(t,1)$-coloring, a contradiction.
Hence, $d$ is even and therefore, $k$ as well. Again with Lemma \ref{gcd at least 3} it follows that there is a $(2t,2)$-coloring.

($\Leftarrow$) Since $(G,\sigma)$ does not have a $(t,1)$-coloring but it has a $(2t,2)$-coloring, it follows that $\chi_c((G,\sigma)) =t = \chi((G,\sigma)) - 1$.
\end{proof}

\begin{theorem} \label{not<}
\begin{enumerate}
\item If $(G,\sigma)$ is antibalanced and not bipartite, then $\chi((G,\sigma)) = 3$, and $\chi_c((G,\sigma)) = 2$.
\item For every even $k \geq 2$, there is a signed graph $(G,\sigma)$ with $\chi((G,\sigma)) -1 = \chi_c((G,\sigma)) = k$.
\end{enumerate}
\end{theorem}

\begin{proof}
1. The mapping $c$ from $V(G)$ to $\mathbb{Z}_3$ with $c(v) = 1$ is a 3-coloring of $(G,\sigma)$. 
Hence, $\chi((G,\sigma)) = 3$, by statement 1. If we consider $c$ as a mapping from $V(G)$ to $\mathbb{Z}_4$, then $c$ is a $(4,2)$-coloring of $(G,\sigma)$. Hence, $\chi_c((G,\sigma)) = 2$. 

2. For $i \in \{1, \dots,n\}$ let $(G_i,\sigma^i)$ be a connected signed graph with at least two vertices and all edges negative. 
Take $(G_1,\sigma^1), \dots, (G_n,\sigma^n)$, and for every $j \in \{1, \dots, n\}$ and every $v \in V(G_j)$ connect $v$ to every vertex of $(\bigcup_{i=1}^n  V(G_i)) - V(G_j)$.
The resulting graph is denoted by $(K_n^*,\sigma_n)$.

\begin{claim} If $n$ is even, then
$\chi((K_n^*,\sigma_n)) = n+1$ and $\chi_c((K_n^*,\sigma_n)) = n$. If $n$ is odd, then $\chi((K_n^*,\sigma_n)) = n+2$ and $\chi_c((K_n^*,\sigma_n)) = n+1$.
\end{claim}

Clearly, the all positive subgraph $K_n^* - N_{\sigma_n}$ has chromatic number $n$. Since for all $i \in \{1, \dots,n\}$ the signed subgraph $(G_i,\sigma^i)$
has only negative edges, and $G_i$ has at least one edge, it follows that all used colors are not self-inverse. Since $n$ is even, it follows that $\chi((K_n^*,\sigma_n)) = n+1$.
Furthermore $c: V(K_n^*) \longrightarrow \mathbb{Z}_{2n}$ with $c(v) = 2i - 1$ if $v \in V(G_i)$ is a $(2n,2)$-coloring of $(K_n^*,\sigma_n)$.
Hence, $\chi_c((K_n^*,\sigma_n)) = n$.

If $n$ is odd, the statement will be proved analogously, and the Claim is proved. Statement 1 of this theorem is the case $n=1$.
\end{proof}

Note, that Lemma \ref{gcd =2} does not apply to the graphs of Theorem \ref{not<} since the cardinality of the set of colors is smaller than the order of the graphs.
It would be of interest whether a statement like Theorem \ref{not<} 2. is also true for odd $k$. 
Furthermore, is there a non-trivial characterization of the signed graphs with $\chi((G,\sigma)) - 1 = \chi_c((G,\sigma))$?

The next theorem shows that if the lower bound in Theorem \ref{bounds} is not attained, then it can be improved. 

\begin{theorem}
	Let $(G,\sigma)$ be a signed graph on $n$ vertices, then either $ \chi((G,\sigma))-1=\chi_c((G,\sigma))$ or $(\chi((G,\sigma))-1)(1+\dfrac{1}{4n-1}) \leq \chi_c((G,\sigma))\leq \chi((G,\sigma))$. 
In particular, if $\chi((G,\sigma))-1\neq \chi_c((G,\sigma))$, then $\chi((G,\sigma)) - \chi_c((G,\sigma)) < 1 - \frac{1}{2n}$.
\end{theorem}

\begin{proof} 
	By Theorem \ref{bounds}, it suffices to show, that  if  $ \chi((G,\sigma))-1\neq \chi_c((G,\sigma))$ then $ (\chi((G,\sigma))-1)(1+\dfrac{1}{4n-1}) \leq \chi_c((G,\sigma)). $
	By Theorem \ref{cor_min}, we may assume that $\chi_c((G,\sigma))=\dfrac{p}{q}$, where $p$ and $q$ are coprime integers and $p\leq 4n$.  
	Then 
\begin{equation} \label{eq}
\chi_c((G,\sigma))-(\chi((G,\sigma))-1)\geq \dfrac{1}{q}=\dfrac{\chi_c((G,\sigma))}{p}\geq \dfrac{\chi_c((G,\sigma))}{4n}.
\end{equation}
	By simplifying  the inequality, we get $$ (\chi((G,\sigma))-1)(1+\dfrac{1}{4n-1}) \leq \chi_c((G,\sigma)). $$

Since $2q<p$, it follows with the first inequality of equation (\ref{eq}) that $\chi((G,\sigma)) - \chi_c((G,\sigma)) < 1 - \frac{1}{2n}$.
\end{proof}

\section{$r$-colorings} \label{equivalence}

\begin{theorem}\label{r-c}
Let $(G,\sigma)$ be a signed graph and $k,d$ be positive integers with $2d\leq k$. $(G,\sigma)$ has a $(2k,2d)$-coloring if and only if $(G,\sigma)$ has a circular $\frac{k}{d}$-coloring.
Furthermore, if $(G,\sigma)$ has a circular $\frac{k}{d}$-coloring such that a common denominator of the used colors is odd, then
$(G,\sigma)$ has a $(k,d)$-coloring.
\end{theorem}

\begin{proof}
We give an analogous proof to the one for unsigned graphs (see Theorem 1 in \cite{Zhu_1992}).

Suppose that $c: V(G)\mapsto \mathbb{Z}_{2k}$ is a $(2k,2d)$-coloring of $(G,\sigma)$.
For each $v \in V(G)$ set $f(v)=\frac{c(v)}{2d}$.
It is easy to verify that $f$ is a circular $\frac{k}{d}$-coloring of $(G,\sigma)$.

On the other side, suppose that $f$ is a circular $r$-coloring of $(G,\sigma)$ with $r=\frac{k}{d}$ and $gcd(k,d)=1$.
Let $S=\{f(v)\colon\ v\in V(G)\}$. The cardinality of $S$ is finite since $G$ is a finite graph. We first show that we can assume that all elements of $S$ are rational numbers.
We will show that each non-rational color can be shifted to a rational color without creating a new pair of colors with distance less than 1. 
Let $s \in S$ and suppose that $s$ is not a rational number. Let $P=P_1, \dots, P_n$ be the longest sequence of pairwise distinct points in $[0,r)$ which satisfies the following constraints:
\begin{itemize}
  \setlength{\itemsep}{0pt}
  \item $s \in P$, and
  \item $\{P_i,r-P_i\}\cap S \neq \emptyset$ and $P_{i+1}=[P_i+1]_r$, where $P_i$ is the element of $P$ in the $i$ place.
\end{itemize}

Define $Q$ to be the sequence consisting of the opposite points of $P$. More precisely, $Q_i=r-P_i$.
Let $\overline{P} = S \cap P$ and $\overline{Q} = S \cap Q$.
Let $\varepsilon$ be a positive real number such that $s+\varepsilon$ is rational.
We shift the colors in $\overline{P}$ together by distance $\varepsilon$ clockwise, and the ones in $\overline{Q}$ together by the same distance anticlockwise. Choose $\varepsilon$ to be small enough.
It is easy to see that this shift is the one required if we can show that the sequences $\overline{P}$ and $\overline{Q}$ contains no common colors. If $\alpha$ is a common color of $\overline{P}$ and $\overline{Q}$, then $s-\alpha$ is an integer and so does $r-s-\alpha$. It follows that $r-2s$ is an
integer, contradicting with the fact that $r$ is a rational number but $s$ not.

Let $m$ be a common denominator of all the colors in $S$. Then the mapping $f'\colon\ V(G)\mapsto \mathbb{Z}_{mk}$ defined as $f'(v)=f(v)md$ is a $(mk,md)$-coloring of $(G,\sigma)$.
Since $m$ can be chosen to be even it follows with Lemma \ref{gcd at least 3} that there is $(2k,2d)$-coloring of $(G,\sigma)$. 
Furthermore, if $m$ is odd, then it follows again with Lemma \ref{gcd at least 3} that $(G,\sigma)$ has a $(k,d)$-coloring.
\end{proof}

The `$(2k,2d)$-coloring' in the previous theorem can not be replaced by `$(k,d)$-coloring' since otherwise there exist counterexamples. The unbalanced triangle is one of the signed graphs that has a circular 2-coloring but has no (2,1)-colorings.

With Theorems \ref{cor_min} and \ref{r-c} we deduce the following statement.

\begin{theorem} If $(G,\sigma)$ is a signed graph, then
$\chi_c((G,\sigma)) = \min \{r :  (G,\sigma)$ has a circular $r$-coloring$\}.$
\end{theorem}
 
\section{Concluding remarks}

First we determine the circular chromatic number of some specific graphs. For $n \geq 3$, let $C_n$ denote the circuit with $n$ vertices. 

\begin{proposition}
Let  $k$ be a positive integer. 

\begin{enumerate}
\item If $(C_{2k+1},\sigma)$ is balanced, then $\chi_c((C_{2k+1},\sigma))=2+\frac{1}{k}$, otherwise $\chi_c((C_{2k+1},\sigma))=2$. Furthermore $\chi((C_{2k+1},\sigma))=3$.
\item $\chi((G,\sigma)) = 2$ if and only if $G$ is bipartite. Furthermore, $\chi((G,\sigma)) = \chi_c((G,\sigma))$ if $G$ is bipartite.
\end{enumerate}
\end{proposition}

\begin{proof}
1. If $(C_{2k+1},\sigma)$ is balanced, then $(C_{2k+1},\sigma)$ is switching equivalent to $(C_{2k+1},+)$, hence, $\chi_c((C_{2k+1},\sigma))=2+\dfrac{1}{k}$. If $(C_{2k+1},\sigma)$ is unbalanced, then $(C_{2k+1},\sigma)$ is switching equivalent to $C_{2k+1}$ which has one negative edge say, $uv$. Thus, we can assign to vertex $u$ and $v$ color 1, and to other vertices colors $3$ and $1$ alternatively. We thereby get a $(4,2)$-coloring of $(C_{2k+1},\sigma)$, i.e., $\chi_c((C_{2k+1},\sigma))=2$. And it is easy to check $(C_{2k+1},\sigma)$ has a $(3,1)$-coloring, but can not be colored properly by two colors, therefore, $\chi((C_{2k+1},\sigma))=3$.

2. If $G$ is bipartite, then it can be colored with colors 0 and 1 and therefore, $\chi((G,\sigma)) = 2$. If $\chi((G,\sigma)) = 2$, then, since both colors are self-inverse in $\mathbb{Z}_2$,
both color classes are independent sets. Hence, $G$ is bipartite. Since $\chi((G,\sigma)), \chi_c((G,\sigma)) \geq 2$, it follows that $\chi((G,\sigma)) = \chi_c((G,\sigma))$ if 
$G$ is bipartite.
\end{proof}

\subsection{Colorings with "signed" colors}

Next, we relate our parameter to the colorings which are considered in \cite{MRS_2014}. The definitions are given in section \ref{Intro}.

\begin{proposition} \label{pm}
If $(G,\sigma)$ is a signed graph, then $\chi_{\pm}((G,\sigma)) - 1 \leq \chi((G,\sigma)) \leq \chi_{\pm}((G,\sigma)) + 1$.
\end{proposition}

\begin{proof} Let $\chi_{\pm}((G,\sigma)) = n$ and $c$ be an $n$-coloring of $(G,\sigma)$ with colors from $M_n$.
 
If $n = 2k+1$, then let 
$\phi : M_{2k+1} \rightarrow \mathbb{Z}_{2k+1}$ with $\phi(t) = t$ if $t \in \{0, \dots, k\}$ and $\phi(t) = 2k+1 + t$ if $t \in \{-k, \dots, -1\}$. Then $c$ is
a $(2k+1)$-coloring of $(G,\sigma)$ with colors from $M_{2k+1}$ if and only if  $\phi \circ c$ is a $(2k+1)$-coloring of $(G,\sigma)$. 
Hence, $\chi((G,\sigma)) \leq \chi_{\pm}((G,\sigma))$.
If $n=2k$, then let $\phi' : M_{2k} \rightarrow \mathbb{Z}_{2k+1}$ with $\phi(t) = t$ if $t \in \{1, \dots, k\}$ and $\phi(t) = 2k+1 + t$ if $t \in \{-k, \dots, -1\}$. 
Then $\phi' \circ c$ is a $(2k+1)$-coloring of $(G,\sigma)$. Hence, $\chi((G,\sigma)) \leq \chi_{\pm}((G,\sigma))+1$.

We analogously deduce that $\chi_{\pm}((G,\sigma)) \leq \chi((G,\sigma)) + 1$.
\end{proof}

The next proposition shows, that the bounds of Proposition \ref{pm} cannot be improved (see Figure 2).

\begin{proposition}
Let $(G,\sigma)$ be a connected signed graph with at least three vertices.
\begin{enumerate}
\item If $(G,\sigma)$ is antibalanced and not bipartite, then $\chi_{\pm}((G,\sigma)) = 2$ and $\chi((G,\sigma)) = 3$.
\item If $(G,\sigma)$ is bipartite but not antibalanced, then $\chi_{\pm}((G,\sigma)) = 3$ and $\chi((G,\sigma)) = 2$.
\end{enumerate}
\end{proposition}

\begin{figure}[hh]
	\centering
	\includegraphics[width=14cm]{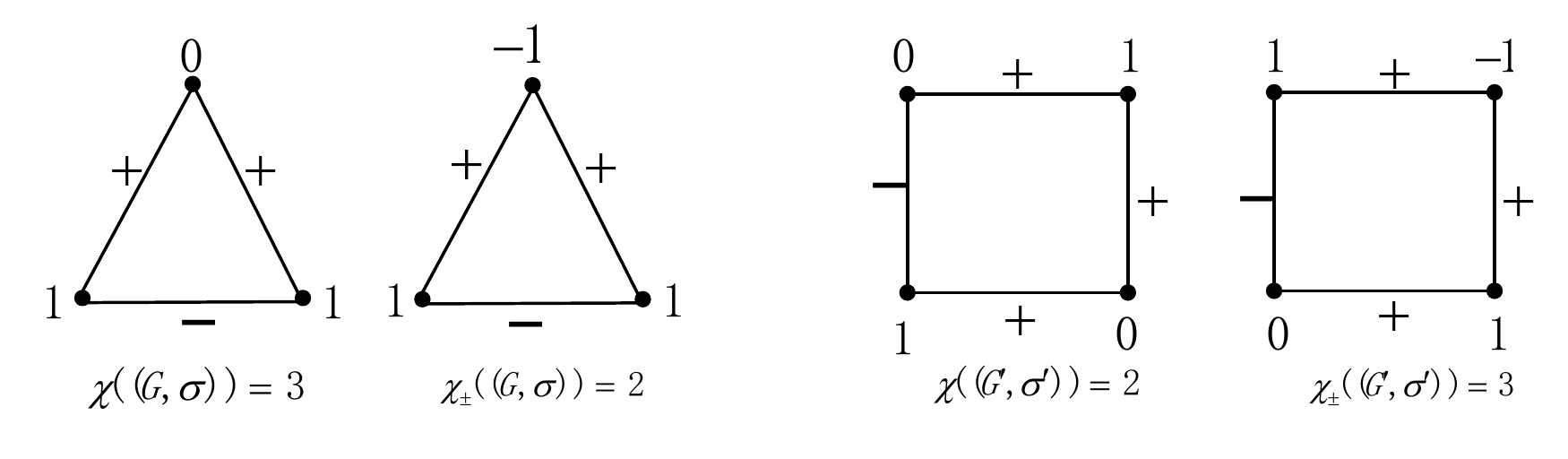}\\
	\caption{relation between $\chi$ and $\chi_{\pm}$}\label{Fig2}
\end{figure}

\end{document}